\documentclass{amsart}
\usepackage{amsfonts,amssymb,amsmath,amsthm}
\usepackage{url}
\usepackage{enumerate}

\usepackage{amsmath,amssymb,amsfonts,amsbsy}
\usepackage{tikz-cd}
\usepackage{amsbsy}
\usepackage{amscd}
\usepackage{wasysym}
\usetikzlibrary{matrix,arrows,decorations.pathmorphing}
\usepackage{graphicx}
\usepackage{float}

\newtheorem{thrm}{Theorem}[section]

\newtheorem{prop}[thrm]{Proposition}
\newtheorem{cor}[thrm]{Corollary}
\theoremstyle{definition}

\newtheorem{remark}[thrm]{Remark}
\numberwithin{equation}{section}

\newcommand{\mc}[1]{\mathcal{#1}}
\newcommand{\e}[1]{\emph{#1}}

\newcommand{\la}{\langle}
\newcommand{\ra}{\rangle}

\newcommand{\rmv}[1]{}

\newcommand{\hs}{\hskip10pt}

\newcommand{\LG}{VN(G)}

\newcommand{\LO}{L^1(G)}
\newcommand{\LOQ}{L^1(\mathbb{G})}

\newcommand{\LOQH}{L^1(\widehat{\mathbb{G}})}

\newcommand{\LTQ}{L^2(\mathbb{G})}
\newcommand{\LTQH}{L^2(\widehat{\mathbb{G}})}

\newcommand{\LI}{L^{\infty}(G)}
\newcommand{\LIQ}{L^{\infty}(\mathbb{G})}

\newcommand{\LIQH}{L^{\infty}(\widehat{\mathbb{G}})}

\newcommand{\LIQHP}{L^{\infty}(\widehat{\mathbb{G}}')}

\newcommand{\BLT}{\mc{B}(L^2(G))}
\newcommand{\BLTQ}{\mc{B}(L^2(\mathbb{G}))}

\newcommand{\TC}{\mc{T}(L^2(G))}
\newcommand{\TCQ}{\mc{T}(L^2(\mathbb{G}))}

\newcommand{\CBTCrr}{\mc{CB}_{\mc{T}_{\rhd}}(\mc{B}(L^2(\mathbb{G})))}

\newcommand{\CBTCrhr}{\mc{CB}_{\mc{T}_{\widehat{\rhd}}}(\mc{B}(L^2(\mathbb{G})))}
\newcommand{\CBTClhr}{_{\mc{T}_{\widehat{\rhd}}}\mc{CB}(\mc{B}(L^2(\mathbb{G})))}
\newcommand{\CBTClhl}{_{\mc{T}_{\widehat{\lhd}}}\mc{CB}(\mc{B}(L^2(\mathbb{G})))}
\newcommand{\CBTCrhl}{\mc{CB}_{\mc{T}_{\widehat{\lhd}}}(\mc{B}(L^2(\mathbb{G})))}
\newcommand{\CBLTQ}{\mc{CB}_{L^{\infty}(\widehat{\mathbb{G}})}^{L^{\infty}(\mathbb{G})}(\mc{B}(L_2(\mathbb{G})))}
\newcommand{\Mcb}{M_{cb}A(G)}

\newcommand{\McbQl}{M_{cb}^l(L^1(\mathbb{G}))}

\newcommand{\McbQr}{M_{cb}^r(L^1(\mathbb{G}))}
\newcommand{\QcbQr}{Q_{cb}^r(L^1(\mathbb{G}))}

\newcommand{\vphi}{\varphi}

\newcommand{\lm}{\lambda}

\newcommand{\Gam}{\Gamma}
\newcommand{\om}{\omega}
\newcommand{\Om}{\Omega}
\newcommand{\ten}{\otimes}
\newcommand{\oten}{\overline{\otimes}}
\newcommand{\hten}{\widehat{\otimes}}

\newcommand{\id}{\textnormal{id}}
\newcommand{\h}[1]{\widehat{#1}}
\newcommand{\wh}[1]{\widehat{#1}}

\newcommand{\Irr}{\mathrm{Irr}(\mathbb{G})}

\providecommand{\norm}[1]{\lVert#1\rVert}

\newcommand{\G}{\mathbb{G}}

\newcommand{\C}{\mathbb{C}}
\newcommand{\N}{\mathbb{N}}

\newcommand{\R}{\mathbb{R}}

\author{Jason Crann}
\address{School of Mathematics and Statistics, Carleton University, Ottawa, ON, Canada K1S 5B6}
\address{Department of Pure Mathematics, University of Waterloo, Waterloo, ON, Canada N2L 3G1}
\address{Institute for Quantum Computing, University of Waterloo, Waterloo, ON, Canada N2L 3G1}
\address{Department of Mathematics and Statistics, University of Guelph, Guelph, ON, Canada N1G 2W1}
\email{jcrann@uwaterloo.ca}
\thanks{This work was completed as part of the author's doctoral thesis, and was supported by an NSERC Canada Graduate Scholarship and a FCRF Joint PhD Scholarship. The author would like to thank Matthias Neufang and Zhong-Jin Ruan for helpful discussions.}

\begin{document}

\title[Amenability and covariant injectivity of LCQG II]
 {Amenability and covariant injectivity of locally compact quantum groups II}

\begin{abstract}
Building on our previous work, we study the non-relative homology of quantum group convolution algebras. Our main result establishes the equivalence of amenability of a locally compact quantum group $\G$ and 1-injectivity of $\LIQH$ as an operator $\LOQH$-module. In particular, a locally compact group $G$ is amenable if and only if its group von Neumann algebra $\LG$ is 1-injective as an operator module over the Fourier algebra $A(G)$. As an application, we provide a decomposability result for completely bounded $\LOQH$-module maps on $\LIQH$, and give a simplified proof that amenable discrete quantum groups have co-amenable compact duals which avoids the use of modular theory and the Powers--St\o rmer inequality, suggesting that our homological techniques may yield a new approach to the open problem of duality between amenability and co-amenability.
\end{abstract}
\maketitle

\section{Introduction}
\label{intro}

The connection between amenability of a locally compact group $G$ and injectivity of its group von Neumann algebra $\LG$ has been a topic of interest in abstract harmonic analysis for decades. Amenability of $G$ entails injectivity of $\LG$ \cite{Gui}, however, the converse is not true, e.g., if $G=SL(n,\R)$ for $n\geq2$; indeed, a result of Connes' \cite[Corollary 7]{Connes}, attributed to Dixmier, states that $\LG$ is injective for any separable connected locally compact group.

In \cite{CN}, we clarified this connection by exploiting the $\TC$-module structure of $\BLT$, showing the equivalence of amenability of a locally compact group $G$ and covariant injectivity of $\LG$, meaning the existence of a conditional expectation $E:\BLT\rightarrow\LG$ commuting with the canonical $\TC$-action \cite[Theorem 4.2]{CN}. We also established a corresponding result at the level of locally compact quantum groups $\G$ and studied the relationship between amenability of $\G$ and relative 1-injectivity of its various operator modules over $\TCQ$.

In this paper we build on results from \cite{CN}, focusing on the non-relative homology of operator modules over $\LOQ$ and $\TCQ$. Our main result states that a locally compact quantum group $\G$ is amenable if and only if $\LIQH$ is 1-injective as an operator $\LOQH$-module. This new homological manifestation of quantum group duality shows that in order to recover properties of $\G$, one should not only consider the von Neumann algebraic structure of $\LIQH$, but rather its \textit{operator module} structure. As an application, we provide a decomposability result in the spirit of Haagerup \cite{Ha2} for completely bounded $\LOQH$-module maps on $\LIQH$. Specifically, if $\G$ is amenable then
$$\mc{CB}_{\LOQH}(\LIQH)=\textnormal{span} \ \mc{CP}_{\LOQH}(\LIQH).$$
Moreover, the proof of the above equality leads to a characterization of the predual $\QcbQr$ of the completely bounded (right) multipliers $\McbQr$ for an arbitrary locally compact quantum group $\G$ as
$$\QcbQr\cong C_0(\G)\hten_{\LOQ}\LOQ,$$
which is new even in the group setting.

Arguably the biggest open problem in abstract harmonic analysis on locally compact quantum groups is the duality between amenability and co-amenability. In the group setting, this is Leptin's theorem \cite{Lep}, which states that a locally compact group $G$ is amenable if and only if its Fourier algebra $A(G)$ has a bounded approximate identity. In the quantum group setting, many partial results have been obtained over the years. Ruan showed that a compact Kac algebra $\G$ is co-amenable if and only if its discrete dual $\widehat{\G}$ is amenable \cite[Theorem 4.5]{Ru4}. This equivalence was later generalized by Tomatsu (and, independently by Blanchard and Vaes) to arbitrary compact quantum groups \cite[Theorem 3.8]{Tom}. Tomatsu's argument relies on the specific modular theory of discrete quantum groups in order to apply the Powers--St\o rmer inequality in a crucial step. As another application of our main result, we give a considerably simplified proof of Tomatsu's theorem which avoids the use of modular theory and the Powers--St\o rmer inequality, suggesting that our homological techniques may provide a new approach to the general duality problem of amenability and co-amenability.

For regular quantum groups $\G$, we obtain a version of our main result at the predual level, showing the equivalence of discreteness of $\G$ and 1-projectivity of $\LOQ$ as an operator module over itself.

The paper is structured as follows. We begin in section 2 with some preliminaries on the homology of operator modules, and include some new results on the relationship between relative and non-relative homology. Section 3 is devoted to a brief overview of the relevant machinery from locally compact quantum groups, their associated operator modules, and completely bounded multipliers. Section 4 outlines the operator module structure of $\BLTQ$ over $\TCQ$ and contains new results which are used in the proof of the main theorem. Section 5 contains the main result of the paper along with its aforementioned applications.

\section{Preliminaries}

Let $\mc{A}$ be a complete contractive Banach algebra. We say that an operator space $X$ is a right \e{operator $\mc{A}$-module} if it is a right Banach $\mc{A}$-module such that the module map $m_X:X\hten\mc{A}\rightarrow X$ is completely contractive, where $\hten$ denotes the operator space projective tensor product. We say that $X$ is \e{faithful} if for every non-zero $x\in X$, there is $a\in\mc{A}$ such that $x\cdot a\neq 0$, and we say that $X$ is \e{essential} if $\la X\cdot\mc{A}\ra=X$, where $\la\cdot\ra$ denotes the closed linear span. We denote by $\mathbf{mod}-\mc{A}$ the category of right operator $\mc{A}$-modules with morphisms given by completely bounded module homomorphisms. Left operator $\mc{A}$-modules and operator $\mc{A}$-bimodules are defined similarly, and we denote the respective categories by $\mc{A}-\mathbf{mod}$ and $\mc{A}-\mathbf{mod}-\mc{A}$. If $X\in\mathbf{mod}-\mc{A}$ is a dual operator space such that the action of $a$ is weak* continuous for all $a\in\mc{A}$, then we say that $X$ is a \e{dual right operator $\mc{A}$-module}. We let $\mathbf{nmod}-\mc{A}$ denote the category of dual right operator $\mc{A}$-modules with morphisms given by weak*-weak* continuous completely bounded module homomorphisms, and similarly for dual left operator $\mc{A}$-modules.

\begin{remark} Regarding terminology, in what follows we will often omit the term ``operator'' when discussing homological properties of operator modules as we will be working exclusively in the operator space category.
\end{remark}

Let $\mc{A}$ be a completely contractive Banach algebra, $X\in\mathbf{mod}-\mc{A}$ and $Y\in\mc{A}-\mathbf{mod}$. The \e{$\mc{A}$-module tensor product} of $X$ and $Y$ is the quotient space $X\hten_{\mc{A}}Y:=X\hten Y/N$, where
$$N=\la x\cdot a\ten y-x\ten a\cdot y\mid x\in X, \ y\in Y, \ a\in\mc{A}\ra,$$
and, again, $\la\cdot\ra$ denotes the closed linear span. It follows that
$$\mc{CB}_{\mc{A}}(X,Y^*)\cong N^{\perp}\cong(X\hten_{\mc{A}} Y)^*,$$
where $\mc{CB}_{\mc{A}}(X,Y^*)$ denotes the space of completely bounded right $\mc{A}$-module maps $\Phi:X\rightarrow Y^*$.
If $Y=\mc{A}$, then clearly $N\subseteq\mathrm{Ker}(m_X)$ where $m_X:X\hten\mc{A}\rightarrow X$ is the module map. If the induced mapping $\widetilde{m}_X:X\hten_{\mc{A}}\mc{A}\rightarrow X$ is a completely isometric isomorphism we say that $X$ is an \e{induced $\mc{A}$-module}. A similar definition applies for left modules. In particular, we say that $\mc{A}$ is \e{self-induced} if $\widetilde{m}_\mc{A}:\mc{A}\hten_{\mc{A}}\mc{A}\cong\mc{A}$ completely isometrically.

Let $\mc{A}$ be a completely contractive Banach algebra and $X\in\mathbf{mod}-\mc{A}$. The identification $\mc{A}^+=\mc{A}\oplus_1\C$ turns the unitization of $\mc{A}$ into a unital completely contractive Banach algebra, and it follows that $X$ becomes a right operator $\mc{A}^+$-module via the extended action
\begin{equation*}x\cdot(a+\lm e)=x\cdot a+\lm x, \ \ \ a\in\mc{A}^+, \ \lm\in\C, \ x\in X.\end{equation*}
Let $C\geq1$. We say that $X$ is \e{relatively $C$-projective} if there exists a morphism $\Phi^+:X\rightarrow X\hten\mc{A}^+$ satisfying $\norm{\Phi^+}_{cb}\leq C$ which is a right inverse to the extended module map $m_X^+:X\hten\mc{A}^+\rightarrow X$. When $X$ is essential, this is equivalent to the existence of a morphism $\Phi:X\rightarrow X\hten\mc{A}$ satisfying $\norm{\Phi}_{cb}\leq C$ and $m_X\circ\Phi=\id_{X}$ by the operator analogue of \cite[Proposition 1.2]{DP}. We say that $X$ is \e{$C$-projective} if for every $Y,Z\in\mathbf{mod}-\mc{A}$, every complete quotient morphism $\Psi:Y\twoheadrightarrow Z$, every morphism $\Phi:X\rightarrow Z$, and every $\varepsilon>0$, there exists a morphism $\widetilde{\Phi}_\varepsilon:X\rightarrow Y$ such that $\norm{\widetilde{\Phi}_\varepsilon}_{cb}< C\norm{\Phi}_{cb}+\varepsilon$  and $\Psi\circ\widetilde{\Phi}_\varepsilon=\Phi$.


%

When $\mc{A}=\C$, the definition of $C$-projectivity coincides with that of a $C$-projective operator space \cite[Definition 3.3]{Blech}. The following relationship between relative projectivity and projectivity appears to be new, and will be used to characterize the 1-projectivity of quantum group convolution algebras.

\begin{prop}\label{p:rel+proj} Let $\mc{A}$ be a completely contractive Banach algebra and let $X\in\mathbf{mod}-\mc{A}$. If $X$ is $C_1$-projective in $\mathbf{mod}-\C$ and is relatively $C_2$-projective in $\mathbf{mod}-\mc{A}$, then $X$ is $C_1C_2$-projective in $\mathbf{mod}-\mc{A}$.\end{prop}

\begin{proof} Let $Y,Z\in\mathbf{mod}-\mc{A}$, let $\Psi:Y\twoheadrightarrow Z$ be a complete quotient morphism and $\Phi:X\rightarrow Z$ be a morphism. By relative $C_2$-projectivity, there exists a morphism $\alpha^+:X\rightarrow X\hten\mc{A}^+$ satisfying $m_X^+\circ\alpha^+=\id_X$ and $\norm{\alpha^+}_{cb}\leq C_2$. Since $X$ is a $C_1$-projective operator space, for every $\varepsilon>0$, there exists a lifting $\Phi_{\varepsilon}:X\rightarrow Y$ satisfying $\Psi\circ\Phi_{\varepsilon}=\Phi$ and $\norm{\Phi_{\varepsilon}}_{cb}<C_1\norm{\Phi}_{cb}+\varepsilon/C_2$. The morphism $(\Phi_{\varepsilon}\ten\id):X\hten\mc{A}^+\rightarrow Y\hten\mc{A}^+$ then satisfies $\norm{\Phi_{\varepsilon}\ten\id}_{cb}< C_1\norm{\Phi}_{cb}+\varepsilon/C_2$, and composing with $\alpha^+$ together with the multiplication $m_Y^+:Y\hten\mc{A}^+\rightarrow Y$, we obtain a morphism $\widetilde{\Phi}_{\varepsilon}:=m_Y^+\circ(\Phi_{\varepsilon}\ten\id)\circ\alpha^+:X\rightarrow Y$ satisfying $\norm{\widetilde{\Phi}_{\varepsilon}}_{cb}< C_1C_2\norm{\Phi}_{cb}+\varepsilon$. Moreover, using the module properties of the relevant morphisms we have
\begin{align*}\Psi\circ\widetilde{\Phi}_{\varepsilon}&=\Psi\circ m_Y^+\circ(\Phi_{\varepsilon}\ten\id)\circ\alpha^+\\
&=m_Z^+\circ(\Psi\ten\id)\circ(\Phi_{\varepsilon}\ten\id)\circ\alpha^+\\
&=m_Z^+\circ(\Phi\ten\id)\circ\alpha^+\\
&=\Phi\circ m_X^+\circ\alpha^+\\
&=\Phi.
\end{align*}
Hence, $X$ is $C_1C_2$-projective.
\end{proof}

Note that the converse of Proposition \ref{p:rel+proj} (when $C_1=C_2=1$) is not true in general as $\mc{A}$ is both $1$-projective and relatively $1$-projective in $\mathbf{mod}-\mc{A}$ for any unital $C^*$-algebra. However, the only $C^*$-algebra which is a 1-projective operator space is $\C$ by \cite[Theorem 3.4]{Blech}.

Given a completely contractive Banach algebra $\mc{A}$ and $X\in\mathbf{mod}-\mc{A}$, there is a canonical completely contractive morphism $\Delta^+:X\rightarrow\mc{CB}(\mc{A}^+,X)$ given by
\begin{equation*}\Delta^+(x)(a)=x\cdot a, \ \ \ x\in X, \ a\in\mc{A}^+,\end{equation*}
where the right $\mc{A}$-module structure on $\mc{CB}(\mc{A}^+,X)$ is defined by
\begin{equation*}(\Psi\cdot a)(b)=\Psi(ab), \ \ \ a\in\mc{A}, \ \Psi\in\mc{CB}(\mc{A}^+,X), \ b\in\mc{A}^+.\end{equation*}
An analogous construction exists for objects in $\mc{A}-\mathbf{mod}$. Let $C\geq 1$. We say that $X$ is \e{relatively $C$-injective} if there exists a morphism $\Phi^+:\mc{CB}(\mc{A}^+,X)\rightarrow X$ such that $\Phi^+\circ\Delta^+=\id_{X}$ and $\norm{\Phi^+}_{cb}\leq C$. When $X$ is faithful, this is equivalent to the existence a morphism $\Phi:\mc{CB}(\mc{A},X)\rightarrow X$ such that $\Phi\circ\Delta=\id_{X}$ and $\norm{\Phi}_{cb}\leq C$ by the operator analogue of \cite[Proposition 1.7]{DP}, where $\Delta(x)(a):=\Delta^+(x)(a)$ for $x\in X$ and $a\in\mc{A}$. We say that $X$ is \e{$C$-injective} if for every $Y,Z\in\mathbf{mod}-\mc{A}$, every completely isometric morphism $\Psi:Y\hookrightarrow Z$, and every morphism $\Phi:Y\rightarrow X$, there exists a morphism $\widetilde{\Phi}:Z\rightarrow X$ such that $\norm{\widetilde{\Phi}}_{cb}\leq C\norm{\Phi}_{cb}$ and $\widetilde{\Phi}\circ\Psi=\Phi$.


Clearly, when $\mc{A}=\C$, the definition of $C$-injectivity coincides with that of a $C$-injective operator space \cite[\S24]{Pi}. For general $\mc{A}$,
the dual $X^*$ of any $X\in\mathbf{mod}-\mc{A}$ has a canonical left $\mc{A}$-module structure, and it follows that $X^*$ is $C$-injective in $\mc{A}-\mathbf{mod}$ whenever $X$ is $C$-projective in $\mathbf{mod}-\mc{A}$ by an operator module version of \cite[Theorem 3.5]{Blech}. Moreover, $X^*$ is relatively $C$-injective in $\mc{A}-\mathbf{nmod}$ if and only if $X$ is relatively $C$-projective in $\mathbf{mod}-\mc{A}$.
The next proposition also appears to be new and will be used in the proof of our main result.



\begin{prop}\label{p:rel+inj} Let $\mc{A}$ be a completely contractive Banach algebra and let $X\in\mathbf{mod}-\mc{A}$. If $X$ is $C_1$-injective in $\mathbf{mod}-\C$ and is relatively $C_2$-injective in $\mathbf{mod}-\mc{A}$, then $X$ is $C_1C_2$-injective in $\mathbf{mod}-\mc{A}$.\end{prop}

\begin{proof} We first show that $\mc{CB}(\mc{A}^+,X)$ is $C_1$-injective in $\mathbf{mod}-\mc{A}$ using the standard argument. To this end, let $Y,Z\in\mathbf{mod}-\mc{A}$, let $\kappa:Y\hookrightarrow Z$ be a completely isometric morphism, and let $\alpha:Y\rightarrow \mc{CB}(\mc{A}^+,X)$ be a morphism.
Define $\beta:Y\rightarrow X$ by $\beta(y)=\alpha(y)(e)$, $y\in Y$, where $e$ is the identity in $\mc{A}^+$. Then $\norm{\beta}_{cb}\leq\norm{\alpha}_{cb}$, and by $C_1$-injectivity of $X$ in $\mathbf{mod}-\C$, there exists an extension $\widetilde{\beta}:Z\rightarrow X$ satisfying $\beta=\widetilde{\beta}\circ\kappa$ and $\norm{\widetilde{\beta}}_{cb}\leq C_1\norm{\beta}_{cb}\leq C_1\norm{\alpha}_{cb}$.
Define $\widetilde{\alpha}:Z\rightarrow\mc{CB}(\mc{A}^+,X)$ by $\widetilde{\alpha}(z)(a)=\widetilde{\beta}(z\cdot a)$, for $z\in Z$, $a\in\mc{A}^+$. Then
$$(\widetilde{\alpha}(z)\cdot a)(b)=\widetilde{\alpha}(z)(ab)=\widetilde{\beta}(z\cdot ab)=\widetilde{\alpha}(z\cdot a)(b)$$
for all $z\in Z$ and $a,b\in\mc{A}$. Thus, $\widetilde{\alpha}$ is a module map extending $\alpha$ such that $\norm{\widetilde{\alpha}}_{cb} \leq \norm{\widetilde{\beta}}_{cb}\leq C_1\norm{\alpha}_{cb}$.

Now, by relative $C_2$-injectivity of $X$ in $\mathbf{mod}-\mc{A}$ there exists a morphism $\Phi^+:\mc{CB}(\mc{A}^+,X)\rightarrow X$ satisfying $\Phi^+\circ\Delta^+=\id_X$ and $\norm{\Phi^+}_{cb}\leq C_2$. Thus, if $Y,Z\in\mathbf{mod}-\mc{A}$ with $\kappa:Y\hookrightarrow Z$ a completely isometric morphism, and $\alpha:Y\rightarrow X$ is a morphism, then we may extend the morphism $\Delta^+\circ\alpha:Y\rightarrow\mc{CB}(\mc{A}^+,X)$ to a morphism $\widetilde{\alpha}:Z\rightarrow\mc{CB}(\mc{A}^+,X)$ with $\norm{\widetilde{\alpha}}_{cb}\leq C_1\norm{\Delta^+\circ\alpha}_{cb}\leq C_1\norm{\alpha}_{cb}$. The morphism $\Phi^+\circ\widetilde{\alpha}:Z\rightarrow X$ satisfies $\Phi^+\circ\widetilde{\alpha}\circ\kappa=\alpha$ and $\norm{\Phi\circ\widetilde{\alpha}}_{cb}\leq C_1C_2\norm{\alpha}_{cb}$, and is therefore the desired extension.
\end{proof}

The converse of Proposition \ref{p:rel+inj} is not true in general (when $C_1=C_2=1$). Indeed, for any unital completely contractive Banach algebra $\mc{A}$ and any $1$-injective operator space $X$, it follows from the proof of Proposition \ref{p:rel+inj} that $\mc{CB}(\mc{A},X)$ is $1$-injective in $\mathbf{mod}-\mc{A}$. This clearly implies relative $1$-injectivity in $\mathbf{mod}-\mc{A}$. However, consider $\mc{A}=B(G)$ and $X=\C$, where $B(G)$ is the Fourier-Stieltjes algebra of a non-amenable discrete group $G$. Since $B(G)$ is the operator dual of the full group $C^*$-algebra $C^*(G)$, we have $\mc{CB}(\mc{A},X)=B(G)^*=C^*(G)^{**}$. If this were a $1$-injective operator space, the group $C^*$-algebra $C^*(G)$ would be nuclear \cite{CE1}, forcing $G$ to be amenable by \cite[Theorem 4.2]{L}.

\begin{remark} Our notions of projectivity and injectivity are closer in spirit to the approach taken in operator space theory \cite{Blech} and the recent approach of Helemskii \cite{Helem} rather than Banach algebra homology, where the related notions are usually studied solely from the relative perspective. In particular, we caution the reader that ``injectivity'' as defined in our previous work \cite{CN} coincides with relative 1-injectivity as defined above. \end{remark}

\section{Locally Compact Quantum Groups}

A \e{locally compact quantum group} is a quadruple $\G=(\LIQ,\Gam,\vphi,\psi)$, where $\LIQ$ is a Hopf-von Neumann algebra with co-multiplication $\Gam:\LIQ\rightarrow\LIQ\oten\LIQ$, and $\vphi$ and $\psi$ are fixed left and right Haar weights on $\LIQ$, respectively \cite{KV2,V}. For every locally compact quantum group $\G$, there exists a \e{left fundamental unitary operator} $W$ on $L_2(\G,\vphi)\ten L_2(\G,\vphi)$ and a \e{right fundamental unitary operator} $V$ on $L_2(\G,\psi)\ten L_2(\G,\psi)$ implementing the co-multiplication $\Gam$ via
\begin{equation*}\Gam(x)=W^*(1\ten x)W=V(x\ten 1)V^*, \ \ \ x\in\LIQ.\end{equation*}
Both unitaries satisfy the \e{pentagonal relation}; that is,
\begin{equation}\label{penta}W_{12}W_{13}W_{23}=W_{23}W_{12}\hs\hs\mathrm{and}\hs\hs V_{12}V_{13}V_{23}=V_{23}V_{12}.\end{equation}
By \cite[Proposition 2.11]{KV2}, we may identify $L_2(\G,\vphi)$ and $L_2(\G,\psi)$, so we will simply use $\LTQ$ for this Hilbert space throughout the paper. We denote by $R$ the unitary antipode of $\G$.

Let $\LOQ$ denote the predual of $\LIQ$. Then the pre-adjoint of $\Gam$ induces an associative completely contractive multiplication on $\LOQ$, defined by
\begin{equation*}\star:\LOQ\hten\LOQ\ni f\ten g\mapsto f\star g=\Gam_*(f\ten g)\in\LOQ.\end{equation*}
The multiplication $\star$ is a complete quotient map from $\LOQ\hten\LOQ$ onto $\LOQ$, implying
\begin{equation*}\la\LOQ\star\LOQ\ra=\LOQ.\end{equation*}
Moreover, $\LOQ$ is always self-induced. The proof is a simple application of \cite[Theorem 2.7]{Vaes}, but we provide the details for the convenience of the reader.

\begin{prop}\label{p:self-induced} Let $\G$ be a locally compact quantum group. Then $\LOQ$ is a self-induced completely contractive Banach algebra.
\end{prop}

\begin{proof} Let $\widetilde{m}:\LOQ\hten_{\LOQ}\LOQ\rightarrow\LOQ$ be the induced multiplication map. Then $\widetilde{m}^*:\LIQ\rightarrow(\LOQ\hten_{\LOQ}\LOQ)^*$ is nothing but the co-multiplication $\Gam$. Since $(\LOQ\hten_{\LOQ}\LOQ)^*\cong N^{\perp}$, where $N\subseteq\LOQ\hten\LOQ$ is the closed linear span of $\{f\star g\ten h-f\ten g\star h\mid f,g,h\in\LOQ\}$, given $X\in(\LOQ\hten_{\LOQ}\LOQ)^*\subseteq\LIQ\oten\LIQ$, it follows that $(\Gam\ten\id)(X)=(\id\ten\Gam)(X)$. Hence, $X\in\Gam(\LIQ)$ by \cite[Theorem 2.7]{Vaes}, and $\widetilde{m}^*$ is surjective. Since $\widetilde{m}^*=\Gam$ is also a complete isometry, the result follows.\end{proof}

For any locally compact quantum group $\G$, the canonical $\LOQ$-bimodule structure on $\LIQ$ is given by
\begin{equation*}f\star x=(\id\ten f)\Gam(x)\hs\hs\mathrm{and}\hs\hs x\star f=(f\ten\id)\Gam(x)\end{equation*}
for $x\in\LIQ$, and $f\in\LOQ$. A \e{left invariant mean on $\LIQ$} is a state $m\in \LIQ^*$ satisfying
\begin{equation}\label{leftinv}\la m,x\star f \ra=\la f,1\ra\la m,x\ra, \ \ \ x\in\LIQ, \ f\in\LOQ.\end{equation}
Right and two-sided invariant means are defined similarly. A locally compact quantum group $\G$ is said to be \e{amenable} if there exists a left invariant mean on $\LIQ$. It is known that $\G$ is amenable if and only if there exists a right (equivalently, two-sided) invariant mean (cf. \cite[Proposition 3]{DQV}). We say that $\G$ is \e{co-amenable} if $\LOQ$ has a bounded left (equivalently, right or two-sided) approximate identity (cf. \cite[Theorem 3.1]{BT}).

The \e{left regular representation} $\lm:\LOQ\rightarrow\BLTQ$ of $\G$ is defined by
\begin{equation*}\lm(f)=(f\ten\id)(W), \ \ \ f\in\LOQ,\end{equation*}
and is an injective, completely contractive homomorphism from $\LOQ$ into $\BLTQ$. Then $\LIQH:=\{\lm(f) : f\in\LOQ\}''$ is the von Neumann algebra associated with the dual quantum group $\h{\G}$. Analogously, we have the \e{right regular representation} $\rho:\LOQ\rightarrow\BLTQ$ defined by
\begin{equation*}\rho(f)=(\id\ten f)(V), \ \ \ f\in\LOQ,\end{equation*}
which is also an injective, completely contractive homomorphism from $\LOQ$ into $\BLTQ$. Then $\LIQHP:=\{\rho(f) : f\in\LOQ\}''$ is the von Neumann algebra associated to the quantum group $\h{\G}'$. It follows that $\LIQHP=\LIQH'$, and the left and right fundamental unitaries satisfy $W\in\LIQ\oten\LIQH$ and $V\in\LIQHP\oten\LIQ$ \cite[Proposition 2.15]{KV2}. Moreover, dual quantum groups always satisfy $\LIQ\cap\LIQH=\LIQ\cap\LIQHP=\C1$ \cite[Proposition 3.4]{VD}.


If $G$ is a locally compact group, we let $\G_a=(\LI,\Gam_a,\vphi_a,\psi_a)$ denote the \e{commutative} quantum group associated with the commutative von Neumann algebra $\LI$, where the co-multiplication is given by $\Gam_a(f)(s,t)=f(st)$, and $\vphi_a$ and $\psi_a$ are integration with respect to a left and right Haar measure, respectively. The dual $\h{\G}_a$ of $\G_a$ is the \e{co-commutative} quantum group $\G_s=(\LG,\Gam_s,\vphi_s,\psi_s)$, where $\LG$ is the left group von Neumann algebra with co-multiplication $\Gam_s(\lm(t))=\lm(t)\ten\lm(t)$, and $\vphi_s=\psi_s$ is Haagerup's Plancherel weight (cf. \cite[\S VII.3]{T2}). Then $L^1(\G_a)$ is the usual group convolution algebra $\LO$, and $L^1(\G_s)$ is the Fourier algebra $A(G)$. It is known that every commutative locally compact quantum group is of the form $\G_a$ \cite[Theorem 2; \S2]{T,VV}. Therefore, every commutative locally compact quantum group is co-amenable, and is amenable if and only if the underlying locally compact group is amenable. By duality, every co-commutative locally compact quantum group is of the form $\G_s$, which is always amenable \cite[Theorem 4]{Re}, and is co-amenable if and only if the underlying locally compact group is amenable, by Leptin's theorem \cite{Lep}.

For a locally compact quantum group $\G$, we let $C_0(\G):=\overline{\h{\lm}(\LOQH)}^{\norm{\cdot}}$ denote the \e{reduced quantum group $C^*$-algebra} of $\G$. We say that $\G$ is \e{compact} if $C_0(\G)$ is a unital $C^*$-algebra, in which case we denote $C_0(\G)$ by $C(\G)$. We say that $\G$ is \e{discrete} if $\LOQ$ is unital. It is well-known that $\G$ is compact if and only if $\h{\G}$ is discrete, and in that case, $\LOQH\cong\bigoplus_{1} \{T_{n_\alpha}(\C)\mid\alpha\in\Irr\}$,
where $T_{n_\alpha}(\C)$ is the space of $n_\alpha\times n_\alpha$ trace-class operators, and $\Irr$ denotes the set of (equivalence classes of) irreducible co-representations of the compact quantum group $\G$ \cite{Wo}. For general $\G$, the operator dual $M(\G):=C_0(\G)^*$ is a completely contractive Banach algebra containing $\LOQ$ as a norm closed two-sided ideal via the map $\LOQ\ni f\mapsto f|_{C_0(\G)}\in M(\G)$ \cite{HNR2}.

We let $C_u(\G)$ be the \e{universal quantum group $C^*$-algebra} of $\G$, and denote the canonical surjective *-homomorphism by $\pi_u:C_u(\G)\rightarrow C_0(\G)$ \cite{K}. The space $C_u(\G)^*$ then has the structure of a unital completely contractive Banach algebra
such that the map $\LOQ\rightarrow C_u(\G)^*$ given by the composition of the inclusion $\LOQ\subseteq M(\G)$ and $\pi_u^*:M(\G)\rightarrow C_u(\G)^*$ is a completely isometric homomorphism, and it follows that $\LOQ$ is a norm closed two-sided ideal in $C_u(\G)^*$ \cite[Proposition 8.3]{K}.

Let $\G$ be a locally compact quantum group. An element $\hat{b}'\in\LIQH'$ is said to be a \e{completely bounded right multiplier} of $\LOQ$ if $\rho(f)\hat{b}'\in\rho(\LOQ)$ for all $f\in\LOQ$ and the induced map
\begin{equation*}m_{\hat{b}'}^r:\LOQ\ni f\mapsto\rho^{-1}(\rho(f)\hat{b}')\in\LOQ\end{equation*}
is completely bounded on $\LOQ$. We let $\McbQr$ denote the space of all completely bounded right multipliers of $\LOQ$, which is a completely contractive Banach algebra with respect to the norm
\begin{equation*}\norm{[\hat{b}'_{ij}]}_{M_n(\McbQr)}=\norm{[m^r_{\hat{b}_{ij}}]}_{cb}.\end{equation*}
Completely bounded left multipliers are defined analogously and we denote by $\McbQl$ the corresponding completely contractive Banach algebra. We now review the relevant properties of completely bounded multipliers, adopting the notation of \cite{JNR}.

Given $\hat{b}'\in\McbQr$, the adjoint $\Theta^r(\hat{b}'):=(m_{\hat{b}'}^r)^*$ defines a normal completely bounded right $\LOQ$-module map on $\LIQ$.
In general, the restriction $\Theta^r(\hat{b}')|_{C_0(\G)}$ leaves $C_0(\G)$ invariant by \cite[Proposition 4.1]{JNR}, and, together with \cite[Proposition 4.2]{JNR}, we have the completely isometric identifications
\begin{equation}\label{e:Mcbidentifications}\Theta^r:\McbQr\cong\mc{CB}^{\sigma}_{\LOQ}(\LIQ)\cong\mc{CB}_{\LOQ}(C_0(\G)).\end{equation}

It is known that $\McbQr$ is a dual operator space \cite[Theorem 3.5]{HNR2}, with predual $\QcbQr$. When $\G=\G_s$ is co-commutative, Haagerup and Kraus gave a representation for elements of $Q_{cb}(L^1(\G_s))=Q_{cb}(A(G))$ as $\Om_{A,\rho}$ for $A\in C^*_\lm(G)\ten_{\min}K_\infty$ and $\rho\in A(G)\hten T_\infty$ \cite[Proposition 1.5]{HK}, where
$$\la\vphi,\Om_{A,\rho}\ra=\la(\Theta^r(\vphi)\ten\id_{K_{\infty}})(A),\rho\ra, \ \ \ \vphi\in M_{cb}A(G);$$
$C^*_\lm(G)$ is the reduced $C^*$-algebra of $G$; the spaces $K_\infty$ and $T_\infty$ denote the compact and trace-class operators on a countably infinite-dimensional Hilbert space, respectively, and $\ten_{\min}$ denotes the minimum tensor product of $C^*$-algebras. This was later generalized to the setting of Kac algbras by Kraus and Ruan \cite[Theorem 3.3]{KR}. Relying upon the general result \cite[Lemma 1.6]{HK}, their argument readily extends to arbitrary locally compact quantum groups.

\begin{prop}\label{p:QcbQ} Let $\G$ be a locally compact quantum group. Then
$$\QcbQr=\{\Om_{A,\rho}\mid A\in C_0(\G)\ten_{\min}K_\infty, \ \rho\in \LOQ\hten T_\infty\},$$
where $\la\hat{b}',\Om_{A,\rho}\ra=\la(\Theta^r(\hat{b}')\ten\id_{K_{\infty}})(A),\rho\ra$, $\hat{b}'\in\McbQr$.
\end{prop}

\section{$\TCQ\curvearrowright\BLTQ$}

Let $\G$ be a locally compact quantum group. The right fundamental unitary $V$ of $\G$ induces a co-associative co-multiplication
\begin{equation*}\Gam^r:\BLTQ\ni T\mapsto V(T\ten 1)V^*\in\BLTQ\oten\BLTQ,\end{equation*}
and the restriction of $\Gam^r$ to $\LIQ$ yields the original co-multiplication $\Gam$ on $\LIQ$. The pre-adjoint of $\Gam^r$ induces an associative completely contractive multiplication on the space of trace-class operators $\TCQ$, defined by
\begin{equation*}\rhd:\TCQ\hten\TCQ\ni\om\ten\tau\mapsto\om\rhd\tau=\Gam^r_*(\om\ten\tau)\in\TCQ.\end{equation*}
Since $\Gam^r$ is a complete isometry, it follows that $\Gam_*^r$ is a complete quotient map, so we have
\begin{equation}\label{e:densityofproducts}\TCQ=\la\TCQ\rhd\TCQ\ra.\end{equation}
Analogously, the left fundamental unitary $W$ of $\G$ induces a co-associative co-multiplication
\begin{equation*}\Gam^l:\BLTQ\ni T\mapsto W^*(1\ten T)W\in\BLTQ\oten\BLTQ,\end{equation*}
and the restriction of $\Gam^l$ to $\LIQ$ is also equal to $\Gam$. The pre-adjoint of $\Gam^l$ induces another associative completely contractive multiplication \begin{equation*}\lhd:\TCQ\hten\TCQ\ni\om\ten\tau\mapsto\om\lhd\tau=\Gam^l_*(\om\ten\tau)\in\TCQ.\end{equation*}

It was shown in \cite[Lemma 5.2]{HNR2} that the pre-annihilator $\LIQ_{\perp}$ of $\LIQ$ in $\TCQ$ is a norm closed two sided ideal in $(\TCQ,\rhd)$ and $(\TCQ,\lhd)$, respectively, and the complete quotient map
\begin{equation}\label{pi}\pi:\TCQ\ni\om\mapsto f=\om|_{\LIQ}\in\LOQ\end{equation}
is an algebra homomorphism from $(\TCQ,\rhd)$, respectively,
$(\TCQ,\lhd)$, onto $\LOQ$.

By \cite[Proposition 2.1]{KV2} the unitary antipode $R$ satisfies $R(x)=\widehat{J}x^*\widehat{J}$, for $x\in\LIQ$. It therefore extends to a *-anti-automorphism (still denoted) $R:\BLTQ\rightarrow\BLTQ$, via $R(T)=\h{J}T^*\h{J}$, $T\in\BLTQ$. The extended antipode maps $\LIQ$ and $\LIQH$ onto $\LIQ$ and $\LIQHP$, respectively, and satisfies the generalized antipode relations; that is,
\begin{equation}\label{R}( R\ten R)\circ\Gam^r=\Sigma\circ\Gam^l\circ R\hs\mathrm{and}
\hs( R\ten R)\circ\Gam^l=\Sigma\circ\Gam^r\circ R,\end{equation}
where $\Sigma$ is the flip map on $\BLTQ\oten\BLTQ$. At the level of $\TCQ$, the relations (\ref{R}) mean
$$ R_*(\om\rhd\tau)= R_*(\tau)\lhd R_*(\om)\hs\mathrm{and}\hs
 R_*(\om\lhd\tau)= R_*(\tau)\rhd R_*(\om)$$
for all $\om,\tau\in\TCQ$. We may therefore pass between the left and right products using $R$, and as a result, we will often focus on the right product $\rhd$ throughout the article.


Since $\LTQ\cong\LTQH$ for any locally compact quantum group $\G$, applying the above construction to the co-multiplication $\h{\Gam}$ on $\LIQH$ yields two \e{dual} products
\begin{align*}\wh{\rhd}&:\TCQ\hten\TCQ\ni\om\ten\tau\mapsto\om\wh{\rhd}\tau=\h{\Gam}^r_*(\om\ten\tau)\in\TCQ,\\
\wh{\lhd}&:\TCQ\hten\TCQ\ni\om\ten\tau\mapsto\om\wh{\lhd}\tau=\h{\Gam}^l_*(\om\ten\tau)\in\TCQ.\end{align*}
This lifting of quantum group convolution to $\TCQ$ allows one to study properties of $\G$ and $\widehat{\G}$ as well as their interactions a \e{single space}. One such interaction was obtained in \cite{KN}, and states that the dual products anti-commute.

\begin{thrm}\cite[Theorem 3.3]{KN} Let $\G$ be a locally compact quantum group. Then for every $\rho,\om,\tau\in\TCQ$ we have
\begin{equation}\label{e:commutationrelation}(\rho \rhd \om) \wh{\rhd} \tau=(\rho \wh{\rhd} \tau) \rhd \om.\end{equation}
\end{thrm}

Equation (\ref{e:commutationrelation}) has an important consequence (Proposition \ref{p:rightimpliesleft}) that will be used in the proof of the main result.

For a locally compact quantum group $\G$, the multiplications $\rhd$ and $\lhd$ define operator $\TCQ$-bimodule structures on $\BLTQ$ such that
for $x\in\LIQ$ and $f=\om|_{\LIQ}$ with $\om\in\TCQ$, we have
\begin{equation}\label{mod1}x\rhd\om=x\lhd\om=x\star f\hs\mathrm{and}\hs\om\lhd x=\om\rhd x=f\star x.\end{equation}
The bimodule actions of $(\TCQ,\rhd)$ and $(\TCQ,\lhd)$ on $\BLTQ$ are therefore liftings of the usual bimodule action of $\LOQ$ on $\LIQ$. For details on these bimodules we refer the reader to \cite{CN,HNR2}. In what follows we denote the algebra of completely bounded right $(\TCQ,\rhd)$-module (respectively, left $(\TCQ,\wh{\rhd})$-module) maps by $\CBTCrr$ (respectively, $\CBTClhr$).

In \cite[Remark 7.4]{HNR2}, the authors observe that for co-amenable $\G$ we have
\begin{equation*}\CBTCrr\subseteq\CBLTQ,\end{equation*}
where $\CBLTQ$ is the algebra of completely bounded $\LIQH$-bimodule maps on $\BLTQ$ that leave $\LIQ$ globally invariant. As a corollary to the commutation relation (\ref{e:commutationrelation}), we can remove the co-amenability hypothesis in the above inclusion using the following ``automatic'' module property.

\begin{prop}\label{p:rightimpliesleft} Let $\G$ be locally compact quantum group. Then
$$\CBTCrr\subseteq \ \CBTClhr.$$
\end{prop}

\begin{proof} Let $\Phi\in\CBTCrr$, and fix $\rho\in\TCQ$ and $T\in\BLTQ$. Then for any $\om,\tau\in\TCQ$, we have
\begin{align*}\la(\rho\wh{\rhd} T) \rhd \tau,\om\ra&=\la\rho\wh{\rhd}T,\tau\rhd\om\ra=\la T,(\tau\rhd\om)\wh{\rhd}\rho\ra\\
&=\la T,(\tau\wh{\rhd}\rho)\rhd\om\ra=\la T\rhd(\tau\wh{\rhd}\rho),\om\ra.\end{align*}
Thus,
\begin{align*}\la\Phi(\rho\wh{\rhd}T),\tau\rhd\om\ra&=\la\Phi(\rho\wh{\rhd}T)\rhd\tau,\om\ra=\la\Phi((\rho\wh{\rhd}T)\rhd\tau),\om\ra\\
&=\la\Phi(T\rhd(\tau\wh{\rhd}\rho)),\om\ra=\la\Phi(T)\rhd(\tau\wh{\rhd}\rho),\om\ra\\
&=\la\Phi(T),(\tau\wh{\rhd}\rho)\rhd\om\ra=\la\Phi(T),(\tau\rhd\om)\wh{\rhd}\rho\ra\\
&=\la\rho\wh{\rhd}\Phi(T),\tau\rhd\om\ra.\end{align*}
By (\ref{e:densityofproducts}) it follows that $\Phi(\rho\wh{\rhd}T)=\rho\wh{\rhd}\Phi(T)$, as required.\end{proof}

\begin{cor}\label{c:inclusion} For any locally compact quantum group $\G$, we have
\begin{equation*}\CBTCrr\subseteq \ \CBLTQ.\end{equation*}
\end{cor}

\begin{proof} Let $\Phi\in\CBTCrr$, and $\hat{x},\hat{y}\in\LIQH$. Then for any $\rho\in\TCQ$ and $T\in\BLTQ$ we have
\begin{align*}(\hat{x}T\hat{y})\rhd\rho&=(\rho\ten\id)V(\hat{x}T\hat{y}\ten1)V^*=(\rho\ten\id)((\hat{x}\ten 1)V(T\ten1)V^*(\hat{y}\ten1))\\
&=(\hat{y}\cdot\rho\cdot\hat{x}\ten\id)V(T\ten 1)V^*=T\rhd(\hat{y}\cdot\rho\cdot\hat{x}).\end{align*}
Thus, for any $\om\in\TCQ$ we obtain
\begin{align*}\la\Phi(\hat{x}T\hat{y}),\rho\rhd\om\ra&=\la\Phi(\hat{x}T\hat{y})\rhd\rho,\om\ra=\la\Phi((\hat{x}T\hat{y})\rhd\rho),\om\ra\\
&=\la\Phi(T\rhd(\hat{y}\cdot\rho\cdot\hat{x})),\om\ra=\la\Phi(T)\rhd(\hat{y}\cdot\rho\cdot\hat{x}),\om\ra\\
&=\la(\hat{x}\Phi(T)\hat{y})\rhd\rho,\om\ra=\la\hat{x}\Phi(T)\hat{y},\rho\rhd\om\ra.\end{align*}
Again by (\ref{e:densityofproducts}), it follows that $\Phi$ is an $\LIQH$-bimodule map on $\BLTQ$.

By Proposition \ref{p:rightimpliesleft} $\Phi\in \ \CBTClhr$, and since $\h{V}\in\LIQ'\oten\LIQH$, for any $x\in\LIQ$ and $\rho\in\TCQ$ we have
$$(\id\ten\rho)\h{V}(\Phi(x)\ten1)\h{V}^*=\rho \wh{\rhd} \Phi(x)=\Phi(\rho \wh{\rhd} x)=\la\rho,1\ra\Phi(x)=(\id\ten\rho)(\Phi(x)\ten1).$$
It follows that $\h{V}(\Phi(x)\ten1)\h{V}^*=\Phi(x)\ten1$, which implies that $\h{\rho}(\hat{f})\Phi(x)=\Phi(x)\h{\rho}(\hat{f})$ for every
$\hat{f}\in\LOQH$. Since $\h{\rho}(\LOQH)$ is weak* dense in $\LIQ'$, we have $\Phi(x)\in\LIQ''=\LIQ$. Thus, $\Phi$ leaves $\LIQ$ globally
invariant, and the claim follows.\end{proof}

In \cite{CN}, we studied the existence of conditional expectations $E:\BLTQ\rightarrow\LIQH$ commuting with the four $\TCQ$-module structures on $\BLTQ$ arising from $\G$. We now complete this picture by studying the four remaining $\TCQ$-module structures on $\BLTQ$ arising from $\h{\G}$. We denote by $\CBTClhl$ and $\CBTCrhl$ the algebra of completely bounded left and right $(\TCQ,\wh{\lhd})$-module
maps on $\BLTQ$, respectively, and similarly for $(\TCQ,\wh{\rhd})$.

\begin{prop}\label{p:rightdual} Let $\G$ be a locally compact quantum group. There exists a conditional expectation $E:\BLTQ\rightarrow\LIQH$ in $\CBTCrhr$ if and only if $\G=\C1$.\end{prop}

\begin{proof} For any $T\in\BLTQ$ and $\rho\in\TCQ$, we have $T \wh{\rhd} \rho\in\LIQH$, so if such a conditional expectation $E$ exists, then $T \wh{\rhd} \rho=E(T \wh{\rhd} \rho)=E(T) \wh{\rhd} \rho$. By density of products (\ref{e:densityofproducts}), it follows that $E(T)=T$. In particular $\BLTQ\subseteq\LIQH$, which entails that $\G=\widehat{\G}=\C1$. The converse is trivial.\end{proof}

\begin{prop} Let $\G$ be a locally compact quantum group. There exists a conditional expectation $E:\BLTQ\rightarrow\LIQH$ in $\CBTClhl$ if and only if $\G=\C1$.
\end{prop}

\begin{proof} Using the extended unitary antipode of $\widehat{\G}$, denoted by $\h{R}$, it follows that $\h{R}\circ E\circ \h{R}$ is a conditional expectation onto $\LIQH$ in $\CBTCrhr$, so the result follows from Proposition \ref{p:rightdual}. \end{proof}

\begin{prop}\label{p:leftdual} Let $\G$ be a locally compact quantum group. There exists a conditional expectation $E:\BLTQ\rightarrow\LIQH$ in $\CBTClhr$ if and only if $\G$ is amenable.\end{prop}

\begin{proof} If $\G$ is amenable, then by \cite[Theorem 4.2]{CN} there exists a conditional expectation onto $\LIQH$ in $\CBTCrr$, which, thanks to Proposition \ref{p:rightimpliesleft}, lies in $\CBTClhr$.

On the other hand, if there exists such a conditional expectation $E$, then for any $x\in\LIQ$ and $\rho\in\TCQ$ we have
$$\rho \wh\rhd E(x)=E(\rho \wh\rhd  x)=\la\rho,1\ra E(x).$$
As in the proof of Corollary \ref{c:inclusion}, this implies $E(x)\in\LIQ\cap\LIQH=\C1$. Hence,
$\G$ is amenable by \cite[Theorem 3]{SV}.\end{proof}

\begin{prop}\label{p:rhl} Let $\G$ be a locally compact quantum group. There exists a conditional expectation $E:\BLTQ\rightarrow\LIQH$ in $\CBTCrhl$ if and only if $\G$ is amenable.\end{prop}

\begin{proof} This follows from Proposition \ref{p:leftdual} using the extended unitary antipode $\h{R}$.\end{proof}

We record the normal version of Proposition \ref{p:rhl} for later use. The proof is left to the reader to establish.

\begin{prop}\label{p:rhlnormal} Let $\G$ be a locally compact quantum group. There exists a normal conditional expectation $E:\BLTQ\rightarrow\LIQH$ in $\CBTCrhl$ if and only if $\G$ is compact.\end{prop}

\section{Main Result \& Applications}

We are now in position to prove the main result of the paper -- the equivalence of amenability of $\h{\G}$ and 1-injectivity of $\LIQ$ as an operator $\LOQ$-module. The theorem also reveals a duality between the left and right $(\TCQ,\rhd)$-module structures on $\BLTQ$: amenability of $\G$ is captured by left injectivity of $\BLTQ$ \cite[Theorem 5.5]{CN}, while amenability of $\h{\G}$ is captured by right injectivity of $\BLTQ$.
\newpage

\begin{thrm}\label{t:amenabilityrightinjectivity} Let $\G$ be a locally compact quantum group. The following conditions are equivalent:
\begin{enumerate}
\item $\h{\G}$ is amenable;
\item $\BLTQ$ is $1$-injective in $\mathbf{mod}-(\TCQ,\rhd)$;
\item $\LIQ$ is $1$-injective in $\mathbf{mod}-\LOQ$.
\end{enumerate}
\end{thrm}

\begin{proof} $(1)\Rightarrow(2)$: By \cite[Proposition 5.8]{CN} amenability of $\h{\G}$ implies that $\BLTQ$ is relatively 1-injective in $\mathbf{mod}-(\TCQ,\rhd)$. Since $\BLTQ$ is a 1-injective operator space, the implication follows from Proposition \ref{p:rel+inj}.

$(2)\Rightarrow(3)$: If $\BLTQ$ is $1$-injective in $\mathbf{mod}-(\TCQ,\rhd)$, there exists a completely contractive morphism $\Phi:\BLTQ\oten\BLTQ\rightarrow\BLTQ$ which is a left inverse to $\Gam^r$, where the pertinent right $(\TCQ,\rhd)$-module structure on $\BLTQ\oten\BLTQ$ is defined by
$$X\RHD\rho=(\rho\ten\id\ten\id)(\Gam^r\ten\id)(X), \ \ \ X\in\BLTQ\oten\BLTQ, \ \rho\in\TCQ.$$
By the proof of Proposition \ref{p:rightimpliesleft} it follows that $\Phi(\rho \wh{\RHD}  X)=\rho \wh{\rhd} \Phi(X)$, where the left module action $\wh{\RHD}$ is given by
$$\rho \wh{\RHD} X=(\id\ten\rho\ten\id)(\h{\Gam}^r\ten\id)(X), \ \ \ X\in\BLTQ\oten\BLTQ, \ \rho\in\TCQ.$$
Furthermore, the proof of Corollary \ref{c:inclusion} entails the invariance
$$\Phi(\LIQ\oten\BLTQ)\subseteq\LIQ.$$
Since $\Gam^l:\BLTQ\rightarrow\LIQ\oten\BLTQ$, the composition $\Phi\circ\Gam^l$ therefore maps into $\LIQ$. Moreover, if $x\in\LIQ$ then $\Phi\circ\Gam^l(x)=\Phi\circ\Gam^r(x)=x$, so that $\Phi\circ\Gam^l$ is a projection of norm one from $\BLTQ$ onto $\LIQ$. Thus, $\LIQ$ is a $1$-injective operator space.

Next, consider the map $\Psi=\Phi|_{\LIQ\oten\LIQ}:\LIQ\oten\LIQ\rightarrow\LIQ$. Since the right $(\TCQ,\rhd)$-module action on $\BLTQ$ restricts to the canonical right $\LOQ$-module action on $\LIQ$, it follows that $\Psi$ is a completely contractive right $\LOQ$-module map such that $\Psi\circ\Gam=\id_{\LIQ}$. Since $\LIQ$ is faithful in $\mathbf{mod}-\LOQ$, this entails the relative $1$-injectivity of $\LIQ$ in $\mathbf{mod}-\LOQ$, and therefore the $1$-injectivity of $\LIQ$ in $\mathbf{mod}-\LOQ$ by Proposition \ref{p:rel+inj}.

$(3)\Rightarrow(1)$: Viewing $\BLTQ$ as a right operator $\LOQ$-module via
\begin{equation*}\label{rightaction}T \lhd f=(f\ten\id)\Gam^l(T)=(f\ten\id)W^*(1\ten T)W, \ \ \ f\in\LOQ, \ T\in\BLTQ,\end{equation*}
1-injectivity of $\LIQ$ in $\mathbf{mod}-\LOQ$ gives an extension of $\id_{\LIQ}$ to a completely contractive morphism $E:\BLTQ\rightarrow\LIQ$. Proposition \ref{p:rhl} then entails the amenability of $\widehat{\G}$.\end{proof}

Analogously, there is a left module version of Theorem \ref{t:amenabilityrightinjectivity} involving the left product $\lhd$, the proof of which follows similarly.

\begin{thrm}\label{t:amenabilityleftinjectivity} Let $\G$ be a locally compact quantum group. The following conditions are equivalent:
\begin{enumerate}
\item $\h{\G}$ is amenable;
\item $\BLTQ$ is $1$-injective in $(\TCQ,\lhd)-\mathbf{mod}$;
\item $\LIQ$ is $1$-injective in $\LOQ-\mathbf{mod}$.
\end{enumerate}
\end{thrm}

In the co-commutative setting, we obtain a new characterization of amenable locally compact groups.

\begin{cor}\label{c:amenabilityinjectivitys} Let $G$ be a locally compact group. The following conditions are equivalent:
\begin{enumerate}
\item $G$ is amenable;
\item $\LG$ is $1$-injective in $\mathbf{mod}-A(G)$;
\item $\LG$ is $1$-injective in $A(G)-\mathbf{mod}$.
\end{enumerate}
\end{cor}

\begin{remark} Corollary \ref{c:amenabilityinjectivitys} highlights the significance of the non-relative homology of $\LG$ as an operator $A(G)$-module. In fact, \e{relative} 1-injectivity of $\LG$ in $\mathbf{mod}-A(G)$ is equivalent to inner amenability of $G$ \cite[Theorem 3.4]{CT}. Related results at the level of quantum groups will appear in subsequent work.\end{remark}

In \cite{Ha2} Haagerup provided an elegant characterization of injective von Neumann algebras via decomposability of completely bounded maps. More specifically, a von Neumann algebra $M$ is injective if and only if $\mc{CB}(M)=\textnormal{span} \ \mc{CP}(M)$, where $\mc{CP}(M)$ is the set of completely positive maps $\Phi:M\rightarrow M$. The next result provides a similar decomposition for $\LOQ$-module maps on $\LIQ$ when $\LIQ$ is $1$-injective in $\mathbf{mod}-\LOQ$.

\begin{prop}\label{p:decomposition} Let $\G$ be a locally compact quantum group. If $\LIQ$ is $1$-injective in $\mathbf{mod}-\LOQ$ (equivalently, $\wh{\G}$ is amenable) then
$$\mc{CB}_{\LOQ}(\LIQ)=\textnormal{span} \ \mc{CP}_{\LOQ}(\LIQ).$$
\end{prop}

\begin{proof} Viewing $M_n(\LIQ)$ as an operator $\LOQ$-module under the amplified action:
$$[x_{ij}]\star f=[x_{ij}\star f], \ \ \ [x_{ij}]\in M_n(\LIQ), \ f\in\LOQ,$$
we claim that $M_n(\LIQ)$ is $1$-injective in $\mathbf{mod}-\LOQ$ for any $n\in\N$. Indeed, the canonical morphism $$\Delta_n:M_n(\LIQ)\rightarrow\mc{CB}(\LOQ,M_n(\LIQ))=M_n(\mc{CB}(\LOQ,\LIQ))$$
is nothing but the $n^{th}$ amplification of $\Delta:\LIQ\rightarrow\mc{CB}(\LOQ,\LIQ)$, so the $n^{th}$ amplification of a completely contractive module left inverse of $\Delta$ (which exists by $1$-injectivity of $\LIQ$) provides a completely contractive module left inverse to $\Delta_n$. Since $M_n(\LIQ)$ is $1$-injective in $\mathbf{mod}-\C$ \cite[Proposition XV.3.2]{T4}, the claim follows from Proposition \ref{p:rel+inj}.

Now, let $\Phi\in\mc{CB}_{\LOQ}(\LIQ)$ be a complete contraction, and consider the Paulsen system $\mc{S}\subseteq M_2(\LIQ)$ defined by
$$\mc{S}=\bigg\{\begin{pmatrix} \alpha1 & x\\ y & \beta1\end{pmatrix}\bigg| \ x,y\in\LIQ, \ \alpha,\beta\in\C\bigg\}.$$
Then $\mc{S}$ is an $\LOQ$-submodule of $M_2(\LIQ)$ and $\Phi$ gives rise to a unital completely positive $\LOQ$-module map $\Phi_{\mc{S}}:\mc{S}\rightarrow M_2(\LIQ)$ via off-diagonalization \cite{Paulsen}:
$$\Phi_{\mc{S}}\begin{pmatrix}\begin{pmatrix} \alpha1 & x\\ y & \beta1\end{pmatrix}\end{pmatrix}=\begin{pmatrix} \alpha1 & \Phi(x)\\ \Phi^*(y) & \beta1\end{pmatrix}, \ \ \ \begin{pmatrix} \alpha1 & x\\ y & \beta1\end{pmatrix}\in\mc{S},$$
where $\Phi^*(y)=\Phi(y^*)^*$, $y\in\LIQ$. By $1$-injectivity of $M_2(\LIQ)$ in $\mathbf{mod}-\LOQ$, the map $\Phi_{\mc{S}}$ extends to a completely contractive $\LOQ$-module map $\widetilde{\Phi}:M_2(\LIQ)\rightarrow M_2(\LIQ)$ such that
$$\widetilde{\Phi}\begin{pmatrix}\begin{pmatrix} 1 & 0\\0 & 1\end{pmatrix}\end{pmatrix}=\begin{pmatrix} 1 & 0\\0 & 1\end{pmatrix}.$$
Hence, $\widetilde{\Phi}$ is completely positive and is of the form
$$\widetilde{\Phi}\begin{pmatrix}\begin{pmatrix} x_{11} & x_{12}\\x_{21} & x_{22}\end{pmatrix}\end{pmatrix}=\begin{pmatrix} \Psi_1(x_{11}) & \Phi(x_{12})\\ \Phi^*(x_{21}) & \Psi_2(x_{22})\end{pmatrix}, \ \ \ [x_{ij}]\in M_2(\LIQ),$$
where $\Psi_i\in\mc{CP}_{\LOQ}(\LIQ)$ is associated to $P_{ii}\circ\widetilde{\Phi}\circ P_{ii}$, and
$$P_{ii}\in\mc{CP}_{\LOQ}(M_2(\LIQ))$$
is the diagonal projection onto the $(i,i)^{th}$ entry for $i=1,2$. By \cite[Proposition 5.4.2]{ER}, it follows that the map
$$\widetilde{\Phi}|_{\LIQ}:\LIQ\ni x\mapsto\begin{pmatrix} \Psi_1(x) & \Phi(x)\\ \Phi^*(x) & \Psi_2(x)\end{pmatrix}\in M_2(\LIQ)$$
is a completely positive $\LOQ$-module map. Thus, via polarization (as in \cite[Proposition 5.4.1]{ER}), it follows that $\Phi\in\textnormal{span} \ \mc{CP}_{\LOQ}(\LIQ)$.\end{proof}

\begin{remark} As the proof of Proposition \ref{p:decomposition} shows, when $\LIQ$ is $1$-injective in $\mathbf{mod}-\LOQ$ we can decompose any $\Phi\in\mc{CB}_{\LOQ}(\LIQ)$ into a linear combination of 4 completely positive $\LOQ$-module maps.\end{remark}


For a locally compact group $G$, it is well-known that $B(G)=\Mcb$ whenever $G$ is amenable \cite[Corollary 1.8]{dH}. Using Proposition \ref{p:decomposition} together with \cite[Theorem 5.2]{D} we can now generalize this implication to arbitrary locally compact quantum groups. We note that the same result was obtained under the a priori stronger assumption that $\G$ is co-amenable \cite[Theorem 4.2]{HNR2}.

\begin{cor}\label{B(G)=McbA(G)} Let $\G$ be a locally compact quantum group. If $\h{\G}$ is amenable then
$$C_u(\G)^*\cong\McbQr.$$\end{cor}

\begin{proof} First, we claim that $\mc{CB}_{\LOQ}(C_0(\G),\LIQ)=\mc{CB}_{\LOQ}(C_0(\G))$. One inclusion is obvious, so let $\Phi\in \mc{CB}_{\LOQ}(C_0(\G),\LIQ)$. Then the restriction of its adjoint $\Phi^*|_{\LOQ}\in \ _{\LOQ}\mc{CB}(\LOQ,M(\G))= \ _{\LOQ}\mc{CB}(\LOQ)$, noting that $\LOQ=\la\LOQ\star\LOQ\ra$ is a closed ideal in $M(\G)$. Hence, $(\Phi^*|_{\LOQ})^*\in\mc{CB}^{\sigma}_{\LOQ}(\LIQ)=\mc{CB}_{\LOQ}(C_0(\G))$ by  \cite[Proposition 4.1]{JNR}. But
$$\la(\Phi^*|_{\LOQ})^*(x),f\ra=\la x,\Phi^*|_{\LOQ}(f)\ra=\la x,\Phi^*(f)\ra=\la\Phi(x),f\ra$$
for all $x\in C_0(\G)$ and $f\in\LOQ$, so $(\Phi^*|_{\LOQ})^*$ is an extension of $\Phi$ which leaves $C_0(\G)$ invariant, hence so too does $\Phi$.

Letting $\mc{I}$ denote the complete isometry
$$\mc{CB}_{\LOQ}(C_0(\G))\ni\Phi\mapsto(\Phi^*|_{\LOQ})^*\in\mc{CB}^{\sigma}_{\LOQ}(\LIQ)$$
and $\mc R:\mc{CB}_{\LOQ}(\LIQ)\rightarrow\mc{CB}_{\LOQ}(C_0(\G),\LIQ)$ the completely contractive restriction map, it follows that $\mc{P}^\sigma:=\mc{I}\circ\mc{R}:\mc{CB}_{\LOQ}(\LIQ)\rightarrow\mc{CB}^\sigma_{\LOQ}(\LIQ)$ is a completely contractive projection onto $\mc{CB}^{\sigma}_{\LOQ}(\LIQ)$. Moreover, $\mc{P}^\sigma$ maps $\mc{CP}_{\LOQ}(\LIQ)$ onto $\mc{CP}^\sigma_{\LOQ}(\LIQ)$.

Since $\widehat{\G}$ is amenable, $\mc{CB}_{\LOQ}(\LIQ)=\textnormal{span} \ \mc{CP}_{\LOQ}(\LIQ)$ by Proposition \ref{p:decomposition}, so given $\Phi\in\mc{CB}^{\sigma}_{\LOQ}(\LIQ)$ there exist $\Phi_i\in\mc{CP}_{\LOQ}(\LIQ)$ $i=1,...,4$ such that
$$\Phi=\frac{1}{4}(\Phi_1-\Phi_2+i(\Phi_3-\Phi_4)).$$
But then
$$\Phi=\mc{P}^\sigma(\Phi)=\frac{1}{4}(\mc{P}^\sigma(\Phi_1)-\mc{P}^\sigma(\Phi_2)+i(\mc{P}^\sigma(\Phi_3)-\mc{P}^\sigma(\Phi_4))),$$
and it follows that $\mc{CB}^\sigma_{\LOQ}(\LIQ)=\textnormal{span} \ \mc{CP}^\sigma_{\LOQ}(\LIQ)$. By \cite[Theorem 5.2]{D}, $\textnormal{span} \ \mc{CP}^\sigma_{\LOQ}(\LIQ)\cong C_u(\G)^*$, so we have
$$\McbQr\cong\mc{CB}^\sigma_{\LOQ}(\LIQ)\cong C_u(\G)^*.$$\end{proof}

The observations in the proof of Corollary \ref{B(G)=McbA(G)} lead to the following new characterization of the predual of $\McbQr$.

\begin{prop}\label{p:QcbQiso} Let $\G$ be a locally compact quantum group. Then
$$\QcbQr\cong C_0(\G)\hten_{\LOQ}\LOQ$$
completely isometrically.\end{prop}

\begin{proof} As noted in the proof of Corollary \ref{B(G)=McbA(G)}, we have
$$\mc{CB}_{\LOQ}(C_0(\G))=\mc{CB}_{\LOQ}(C_0(\G),\LIQ).$$
Thus, $\Theta^r:\McbQr\cong\mc{CB}_{\LOQ}(C_0(\G),\LIQ)$
completely isometrically \cite[Proposition 4.1]{JNR}. We need to show that $\Theta^r$ is a weak*-weak* homeomorphism. Since
$$\mc{CB}_{\LOQ}(C_0(\G),\LIQ)\cong(C_0(\G)\hten_{\LOQ}\LOQ)^*$$
weak* homeomorphically, and $\Theta^r$ is a completely isometric isomorphism, it suffices to show that $\Theta^r$ is weak* continuous on bounded sets (see \cite[Lemma 10.1]{D4}). Let $(\hat{b}'_i)_{i\in I}$ be a bounded net in $\McbQr$ converging weak* to $\hat{b}'$. By Proposition \ref{p:QcbQ}, for any $A\in C_0(\G)\ten_{\min}K_{\infty}$ and $\rho\in\LOQ\hten T_\infty$ we have $\Om_{A,\rho}\in\QcbQr$, where
$$\la\hat{a}',\Om_{A,\rho}\ra=\la(\Theta^r(\hat{a}')\ten\id_{K_{\infty}})(A),\rho\ra, \ \ \ \hat{a}'\in\McbQr.$$
Then $(\Theta^r(\hat{b}_i'))_{i\in I}$ converges point weak* to $\Theta^r(\hat{b}')$ in $\mc{CB}(C_0(\G),\LIQ)$. Letting
$$q:C_0(\G)\hten\LOQ\twoheadrightarrow C_0(\G)\hten_{\LOQ}\LOQ$$
be the quotient map, and viewing $\Theta^r(\hat{b}_i')\in(C_0(\G)\hten_{\LOQ}\LOQ)^*$,
the density of the image $q(C_0(\G)\ten\LOQ)$ of the algebraic tensor product $C_0(\G)\ten\LOQ$ in $C_0(\G)\hten_{\LOQ}\LOQ$, together with the boundedness of $(\Theta^r(\hat{b}_i'))_{i\in I}$ imply that $(\Theta^r(\hat{b}_i'))_{i\in I}$ converges weak* to $\Theta^r(\hat{b}')$ in $(C_0(\G)\hten_{\LOQ}\LOQ)^*$.\end{proof}

\begin{remark} The identification of $\QcbQr$ in Proposition \ref{p:QcbQiso} is new even in the co-commutative case, that is, for any locally compact group $G$ we have
$$Q_{cb}(G)\cong C^*_{\lm}(G)\hten_{A(G)}A(G)$$
completely isometrically.\end{remark}

For our final application, we now give a simplified proof of the fact that amenability of a discrete quantum group implies co-amenability of its compact dual.

\begin{thrm}\label{t:simplified} A compact quantum group $\G$ is co-amenable if and only if $\widehat{\G}$ is amenable.\end{thrm}

\begin{proof} Co-amenability of $\G$ always implies amenability of $\h{\G}$ \cite[Theorem 3.2]{BT}, so assume that $\widehat{\G}$ is amenable. By Theorem \ref{t:amenabilityrightinjectivity} we know that $\LIQ$ is 1-injective in $\mathbf{mod}-\LOQ$. Let $\Phi:\LIQ\oten\LIQ\rightarrow\LIQ$ be a completely contractive left inverse to $\Gam$ which is a right $\LOQ$-module map. As a unital complete contraction, $\Phi$ is completely positive and $\Phi|_{C(\G)\ten_{\mathrm{min}}C(\G)}\neq0$ since $C(\G)$ is unital. By \cite[Theorem 3.3]{BT} we also know that $C(\G)$ is nuclear, so let $(\Psi_a)_{a\in A}$ be a net of finite-rank, unital completely positive maps converging to $\id_{C(\G)}$ in the point-norm topology. For $a\in A$, consider the unital completely positive map $\Phi_a:C(\G)\rightarrow C(\G)$ given by
$$\Phi_a=\Phi\circ(\id\ten\Psi_a)\circ\Gam|_{C(\G)}.$$
The fact that $\Phi_a$ maps into $C(\G)$ follows from the density of
$\Gam(C(\G))(C(\G)\ten 1)$ in $C(\G)\ten_{\mathrm{min}}C(\G)$ (see \cite[Corollary 6.11]{KV1}) together with the $\Gam(\LIQ)$-module property of $\Phi$ (see the proof of \cite[Theorem 5.5]{CN}). However, the invariance $\Phi_a(C(\G))\subseteq C(\G)$ will be a byproduct of the following argument.

Since $\Psi_a$ is finite rank, there exist $x^a_1,...,x^a_{n_a}\in C(\G)$ and $\mu^a_1,...,\mu^a_{n_a}\in M(\G)$ such that
$$\Psi_a(x)=\sum_{n=1}^{n_a}\la\mu^a_n,x\ra x^a_n, \ \ \ x\in C(\G), \ a\in A.$$
For each $a\in A$, and $1\leq n\leq n_a$, let $\Phi_{(a,n)}:C(\G)\rightarrow C(\G)$ be defined by
$$\Phi_{(a,n)}(x)=\Phi(x\ten x^a_n), \ \ \ x\in C(\G).$$
Then $\Phi_{(a,n)}$ is completely bounded with $\norm{\Phi_{(a,n)}}_{cb}\leq\norm{x^a_n}_{C(\G)}$, and is a right $\LOQ$-module map. Hence, $\Phi_{(a,n)}\in\mc{CB}_{\LOQ}(C(\G))=\Theta^r(\McbQr)$. Since $\McbQr=C_u(\G)^*$ by Corollary \ref{B(G)=McbA(G)}, there exist $\nu^a_n\in C_u(\G)^*$ such that $\Phi_{(a,n)}=\Theta^r(\nu^a_n)$.

Let $x\in C_0(\G)$. Then
\begin{align*}\Phi_a(x)&=\Phi((\id\ten\Psi_a)(\Gam(x)))\\
&=\sum_{n=1}^{n_a}\Phi(((\id\ten\mu^a_n)\Gam(x))\ten x^a_n)\\
&=\sum_{n=1}^{n_a}\Phi_{(a,n)}(((\id\ten\mu^a_n)\Gam(x)))\\
&=\sum_{n=1}^{n_a}\Theta^r(\nu^a_n)(\Theta^r(\mu^a_n)(x))\\
&=\sum_{n=1}^{n_a}\Theta^r(\nu^a_n\star\mu^a_n)(x)\\
&=\Theta^r\bigg(\sum_{n=1}^{n_a}\nu^a_n\star\mu^a_n\bigg)(x).
\end{align*}
Letting $\mu_a=\sum_{n=1}^{n_a}\nu^a_n\star\mu^a_n$, we obtain $\Phi_a=\Theta^r(\mu_a)$. Moreover, we have $\mu_a\in M(\G)$ as $M(\G)$ is a two-sided ideal in $C_u(\G)^*$. Since $\Theta^r$ is an isometry,
$$\norm{\mu_a}_{M(\G)}=\norm{\Theta^r(\mu_a)}_{cb}=\norm{\Phi_a}_{cb}=1, \ \ \ a\in A,$$
and since $\Phi_a$ converges to $\id_{C(\G)}$ in the point-norm topology it follows that
$$\mu_a\star x=\Theta^r(\mu_a)(x)\rightarrow x, \ \ \  x\in C(\G).$$
Let $\mu$ be a weak* cluster point of $(\mu_a)_{a\in A}$ in the unit ball of $M(\G)=C(\G)^*$. Then $\mu$ is a right identity of $M(\G)$. The restricted unitary antipode $R$ maps $C(\G)$ into $C(\G)$ and satisfies $R^*(\mu\star\nu)=R^*(\nu)\star R^*(\mu)$ for all $\nu\in M(\G)$. Hence, $R^*(\mu)$ is a left identity of $M(\G)$. It follows that $\varepsilon:=\mu+R^*(\mu)-\mu\star R^*(\mu)$ is an identity for $M(\G)$. Hence, $\G$ is co-amenable by \cite[Theorem 3.1]{BT}.\end{proof}

Given a completely contractive Banach algebra $\mc{A}$ with a contractive approximate identity, any essential module $X\in\mathbf{mod}-\mc{A}$ is induced by \cite[Proposition 6.4]{D4}. Since a locally compact quantum group $\G$ is co-amenable if and only if $\LOQ$ has a contractive approximate identity \cite[Theorem 2]{HNR1}, the next proposition supports the idea that our methods may be applicable to the general duality problem of amenability and co-amenability.

\begin{prop} Let $\G$ be a locally compact quantum group for which the dual $\widehat{\G}$ is amenable. Then for any closed right ideal $I\unlhd\LOQ$, the multiplication map yields a completely isometric isomorphism
$$\widetilde{m}_I:I\hten_{\LOQ}\LOQ\cong\la I\star\LOQ\ra.$$
In particular, if $I$ is essential then $I\hten_{\LOQ}\LOQ\cong I$, that is, $I$ is an induced $\LOQ$-module.\end{prop}

\begin{proof} First, note that for any self-induced completely contractive Banach algebra $\mc{A}$ and any closed right ideal $J\unlhd\mc{A}$, we have $\widetilde{m}_{\mc{A}/J}:(\mc{A}/J)\hten_{\mc{A}}\mc{A}\cong\mc{A}/\la J\cdot\mc{A}\ra$ completely isometrically. Indeed, identifying $(\mc{A}/\la J\cdot\mc{A}\ra)^*=\la J\cdot\mc{A}\ra^{\perp}\subseteq\mc{A}^*$, it follows that
$$(\widetilde{m}_{\mc{A}/J})^*:\la J\cdot\mc{A}\ra^{\perp}\rightarrow((\mc{A}/J)\hten_{\mc{A}}\mc{A})^*=(N_{\mc{A}/J})^\perp$$
is equal to $(\widetilde{m}_{\mc{A}})^*|_{\la J\cdot\mc{A}\ra^{\perp}}$. In particular, $(\widetilde{m}_{\mc{A}/J})^*$ is a complete isometry. Letting $q:\mc{A}\twoheadrightarrow\mc{A}/J$ be the complete quotient map, if $X\in (N_{\mc{A}/J})^{\perp}$ then $(q^*\ten\id)(X)\in (N_{\mc{A}})^\perp$, so there exists $F\in\mc{A}^*$ such that $(q^*
\ten\id)(X)=(\widetilde{m}_{\mc{A}})^*(F)$ as $\mc{A}$ is self-induced. Clearly, $F\in\la J\cdot\mc{A}\ra^{\perp}$, so $(\widetilde{m}_{\mc{A}/J})^*$ is also surjective.

Since $\widehat{\G}$ is amenable, by Theorem \ref{t:amenabilityrightinjectivity} $\LIQ$ is 1-injective in $\mathbf{mod}-\LOQ$. Then for every 1-exact sequence of right $\mc{A}$-modules
$$0\rightarrow Y\hookrightarrow Z\twoheadrightarrow Z/Y\rightarrow 0,$$
the induced sequence
\begin{equation*}\label{e:exact}0\rightarrow\mc{CB}_{\LOQ}(Z/Y,\LIQ)\hookrightarrow\mc{CB}_{\LOQ}(Z,\LIQ)\twoheadrightarrow\mc{CB}_{\LOQ}(Y,\LIQ)\rightarrow 0\end{equation*}
is 1-exact, where 1-exactness refers to an exact sequence of morphisms such that the injection ($\hookrightarrow$) is a complete isometry and the surjection $(\twoheadrightarrow)$ is a complete quotient map. Taking the pre-adjoint of the above sequence we obtain the 1-exact sequence
$$0\rightarrow Y\hten_{\LOQ}\LOQ\hookrightarrow Z\hten_{\LOQ}\LOQ\twoheadrightarrow Z/Y\hten_{\LOQ}\LOQ\rightarrow 0.$$
In particular, take $Y=I$ and $Z=\LOQ$, and consider the commutative diagram:
\begin{equation*}
\begin{CD}
I\hten_{\LOQ}\LOQ @>>> \LOQ\hten_{\LOQ}\LOQ @>>> (\LOQ/I)\hten_{\LOQ}\LOQ\\
@VVV @VVV @VVV\\
\la I\star\LOQ\ra @>>> \LOQ @>>> \LOQ/\la I\star\LOQ\ra
\end{CD}
\end{equation*}
As $\LOQ$ is self-induced, the last two columns are completely isometric isomorphisms, and since both rows are 1-exact, it follows that $$\widetilde{m}_I:I\hten_{\LOQ}\LOQ\cong\la I\star\LOQ\ra$$
completely isometrically.\end{proof}

A locally compact quantum group $\G$ is said to be \e{regular} if
$$\mc{K}(\LTQ)=\la (\id\ten\om)(\sigma V)\mid\om\in\TCQ\ra,$$
where $\mc{K}(\LTQ)$ denotes the ideal of compact operators on $\LTQ$, $\sigma$ denotes the flip map on $\LTQ\ten\LTQ$, and as usual, $\la\cdot\ra$ denotes the closed linear span. For example, Kac algebras are regular, as well as discrete and compact quantum groups (see \cite{HNR4}). Under the assumption of regularity, we now obtain a version of Theorem \ref{t:amenabilityrightinjectivity} at the predual level.

\begin{thrm}\label{t:projectivity} Let $\G$ be a locally compact quantum group. Consider the following conditions:
\begin{enumerate}
\item $\h{\G}$ is compact (equivalently, $\G$ is discrete);
\item $\TCQ$ is relatively 1-projective in $(\TCQ,\rhd)-\mathbf{mod}$;
\item $\LOQ$ is 1-projective in $\LOQ-\mathbf{mod}$.
\end{enumerate}
Then $(1)\Leftrightarrow(2)\Rightarrow(3)$, and when $\G$ is regular, the conditions are equivalent.
\end{thrm}

\begin{proof} The implication $(1)\Rightarrow(2)$ follows by an argument similar to the proof of \cite[Proposition 5.8]{CN} using a normal two-sided invariant mean $\h{m}'$ on $\LIQHP$, which exists by compactness. The implication $(2)\Rightarrow(1)$ follows similarly to Theorem \ref{t:amenabilityrightinjectivity}, giving the relative 1-injectivity of $\LIQ$ in $\mathbf{nmod}-\LOQ$ and the existence of a \e{normal} conditional expectation $E:\BLTQ\rightarrow\LIQ$. Viewing $\BLTQ$ as a right $\LOQ$-module under the $\lhd$-action, the relative 1-injectivity of $\LIQ$ in $\mathbf{nmod}-\LOQ$ implies the existence of a normal condition expectation $P:\BLTQ\rightarrow\LIQ$ that is a right $\LOQ$-module map (see \cite[Lemma 3.5]{Wh} for details). Thus, $\h{\G}$ is compact by Proposition \ref{p:rhlnormal}.

$(2)\Rightarrow(3)$: By the above we know that $\h{\G}$ is compact and $\LIQ$ is relatively 1-injective in $\mathbf{nmod}-\LOQ$, which implies that $\LOQ$ is relatively 1-projective in $\LOQ-\mathbf{mod}$. By discreteness of $\G$ we have $\LOQ\cong\bigoplus_{1} \{T_{n_\alpha}(\C)\mid\alpha\in\mathrm{Irr}(\h{\G})\}$, where $T_{n_\alpha}(\C)$ is the space of $n_\alpha\times n_\alpha$ trace-class operators. Hence, $\LOQ$ is 1-projective in $\C-\mathbf{mod}$ by \cite[Proposition 3.6, Proposition 3.7]{Blech}. The left version of Proposition \ref{p:rel+proj} then entails the 1-projectivity of $\LOQ$ in $\LOQ-\mathbf{mod}$.

Now, suppose that $\G$ is regular. Considering again the right $\LOQ$-module structure on $\BLTQ$ given by the $\lhd$-action, it follows from \cite[Corollary 3.6]{HNR4} that $\mc{K}(\LTQ)$ is an essential $\LOQ$-submodule of $\BLTQ$, that is, $\mc{K}(\LTQ)=\la\mc{K}(\LTQ)\lhd\LOQ\ra$. We show $(3)\Rightarrow (1)$.

Since the multiplication $m_{\LOQ}:\LOQ\hten\LOQ\rightarrow\LOQ$ is a complete quotient morphism and $\LOQ$ is 1-projective in $\LOQ-\mathbf{mod}$, for every $\varepsilon>0$ there exists a morphism $\Phi_{\varepsilon}:\LOQ\rightarrow\LOQ\hten\LOQ$ satisfying $m_{\LOQ}\circ\Phi_{\varepsilon}=\id_{\LOQ}$ and $\norm{\Phi_{\varepsilon}}_{cb}< 1+\varepsilon$. Moreover, we know that $\LIQ$ is 1-injective in $\mathbf{mod}-\LOQ$ as the dual a 1-projective module. Thus, $\McbQr\cong C_u(\G)^*$ by Corollary \ref{B(G)=McbA(G)}, and $\LIQ$ is a 1-injective operator space. Hence, $\LIQ$ is semi-discrete \cite{CE1,Connes,EL}, so there exits a net $(\Psi_i)_{i\in I}$ of normal finite-rank complete contractions $\Psi_i:\LIQ\rightarrow\LIQ$ converging to $\id_{\LIQ}$ in the point weak* topology.  Using the \e{normal} completely bounded morphism $\Phi_{\varepsilon}^*:\LIQ\oten\LIQ\rightarrow\LIQ$ which is a left inverse of $\Gam$, one can argue in a similar manner to Theorem \ref{t:simplified} by averaging the normal finite-rank maps $\Psi_i$ into multipliers and use the fact that $\LOQ$ is a two-sided ideal in $C_u(\G)^*$ to obtain a bounded net $(f_i)_{i\in I}$ in $\LOQ$ satisfying $f\star f_i-f\rightarrow 0$ weakly for all $f\in\LOQ$. The standard convexity argument then yields a bounded right approximate identity for $\LOQ$, and $\G$ is necessarily co-amenable.

Now, since $\pi:(\TCQ,\lhd)\rightarrow\LOQ$ is a complete quotient morphism, for any $\varepsilon>0$ it also has a right inverse morphism $\Psi_{\varepsilon}:\LOQ\rightarrow\TCQ$ with $\norm{\Psi_{\varepsilon}}_{cb}<1+\varepsilon$. Then $\Psi_{\varepsilon}^*:\BLTQ\rightarrow\LIQ$ is a normal completely bounded right $(\TCQ,\lhd)$-module projection onto $\LIQ$. Since $\LOQ$ has a contractive approximate identity and $\mc{K}(\LTQ)$ is an essential $\LOQ$-module, we know that $\mc{K}(\LTQ)$ is induced, that is,
$$\widetilde{m}_{\mc{K}(\LTQ)}:\mc{K}(\LTQ)\hten_{\LOQ}\LOQ\rightarrow\mc{K}(\LTQ)$$
is a completely isometric isomorphism. Hence, so too is its dual
$$(\widetilde{m}_{\mc{K}(\LTQ)})^*:\mc{T}(\LTQ)\cong\mc{CB}_{\LOQ}(\mc{K}(\LTQ),\LIQ).$$
Then $\Psi_{\varepsilon}^*|_{\mc{K}(\LTQ)}\in\mc{CB}_{\LOQ}(\mc{K}(\LTQ),\LIQ)=(\widetilde{m}_{\mc{K}(\LTQ)})^*(\mc{T}(\LTQ))$, so there exists $\rho\in\TCQ$ satisfying $(\widetilde{m}_{\mc{K}(\LTQ)})^*(\rho)=\Psi_{\varepsilon}^*|_{\mc{K}(\LTQ)}$. Then for all $y\in\mc{K}(\LTQ)$ and $f\in\LOQ$ we have
$$\la\Psi_{\varepsilon}^*|_{\mc{K}(\LTQ)}(y),f\ra=\la(\widetilde{m}_{\mc{K}(\LTQ)})^*(\rho)(y),f\ra=\la\rho,y\lhd f\ra=\la\rho\lhd y,f\ra.$$
By weak* density of $\mc{K}(\LTQ)$ in $\BLTQ$, we obtain $\Psi_{\varepsilon}^*(T)=\rho\lhd T$ for all $T\in\BLTQ$. In particular,
$$\pi(\rho)\star x=\rho\lhd x=\Psi_{\varepsilon}^*(x)=x$$
for all $x\in\LIQ$ as $\Psi_{\varepsilon}^*$ is a projection. Then $\pi(\rho)$ is a right identity for $\LOQ$, and using the unitary antipode $R$ as in Theorem \ref{t:simplified} we may construct a two-sided identity for $\LOQ$, that is, $\G$ is discrete, whence $\h{\G}$ is compact.\end{proof}

Analogously, there is a right module version of Theorem \ref{t:projectivity}.

\begin{thrm}\label{t:rightprojectivity} Let $\G$ be a locally compact quantum group. Consider the following conditions:
\begin{enumerate}
\item $\h{\G}$ is compact (equivalently, $\G$ is discrete);
\item $\TCQ$ is relatively 1-projective in $\mathbf{mod}-(\TCQ,\lhd)$;
\item $\LOQ$ is 1-projective in $\mathbf{mod}-\LOQ$.
\end{enumerate}
Then $(1)\Leftrightarrow(2)\Rightarrow(3)$, and when $\G$ is regular, the conditions are equivalent.\end{thrm}

\begin{remark} It is not clear at this time whether we can replace relative 1-projectivity of $\TCQ$ with 1-projectivity of $\TCQ$ in condition (2) of Theorems \ref{t:projectivity} and \ref{t:rightprojectivity}. However, one \e{cannot} replace 1-projectivity of $\LOQ$ with relative 1-projectivity of $\LOQ$ in condition (3) of Theorems \ref{t:projectivity} and \ref{t:rightprojectivity}, as, for example, $\LO$ is always relatively 1-projective for any locally compact group $G$ (see \cite[Theorem 2.4]{DP}).\end{remark}

\begin{remark} Combining \cite[Theorem 3.12]{Blech} with \cite[Corollary 7]{Kal}, it follows that a regular quantum group $\G$ is discrete if and only if $\LOQ$ is a 1-projective operator space. Theorem \ref{t:projectivity} therefore yields the following equivalence for regular quantum groups: $\LOQ$ is 1-projective in $\C-\mathbf{mod}$ if and only if $\LOQ$ is 1-projective in $\LOQ-\mathbf{mod}$.\end{remark}



\end{document}